\documentclass[11pt]{amsart}
\usepackage{enumerate}
\usepackage{amsthm}
\usepackage{amsmath,amssymb,latexsym,amscd}
\usepackage{tikz}

\makeatletter
\@namedef{subjclassname@2010}{%
  \textup{2010} Mathematics Subject Classification}
\makeatother
\theoremstyle{definition}
\newtheorem{theorem}{Theorem}[section]

\newtheorem{proposition}[theorem]{Proposition}

\theoremstyle{definition}
\theoremstyle{definition}

\newtheorem{example}[theorem]{Example}
\usepackage{indentfirst}

\DeclareMathOperator{\im}{im}
\def\Z{\mathbb{Z}}

\def\S{\mathbb{S}}

\begin{document}
\baselineskip=17pt
\title[]{Finitistic spaces with  orbit space a  product of  projective spaces}
\author[Anju Kumari and Hemant Kumar Singh]{ Anju Kumari and Hemant Kumar Singh}
\address{{\bf Anju Kumari}, 
	Department of Mathematics, University of Delhi,
	Delhi -- 110007, India.}
\email{anjukumari0702@gmail.com}
\address{{\bf Hemant Kumar Singh}, 
	Department of Mathematics, University of Delhi,
	Delhi -- 110007, India.}
\email{hemantksingh@maths.du.ac.in}

\date{}
\thanks{This paper is supported by the Science and Engineering Research Board (Department of Science and Technology, Government of India) with reference number- EMR/2017/002192}
\begin{abstract}
	Let $G=\Z_2$ act freely on a finitistic space $X$. If the mod 2 cohomology of $X$ is isomorphic to the real projective space $\mathbb{RP}^{2n+1}$ (resp. complex projective space $\mathbb{CP}^{2n+1}$) then the mod 2 cohomology of orbit spaces  of  these free actions are $\mathbb{RP}^1\times \mathbb{CP}^n$ (resp. $\mathbb{RP}^2\times \mathbb{HP}^n$) \cite{Hemant2008}. In this paper, we have discussed  converse of these result. We have showed that  if the mod 2 cohomology of the orbit space $X/G$ is $ \mathbb{RP}^1\times \mathbb{CP}^n$ (resp.  $ \mathbb{RP}^2\times\mathbb{HP}^n$) then the mod 2 cohomology of $X$ is $\mathbb{RP}^{2n+1}$ or $\mathbb{S}^1\times \mathbb{CP}^n$ (resp. $\mathbb{CP}^{2n+1}$ or $\mathbb{S}^2\times \mathbb{HP}^n$).

	A partial converse of  free involutions on the product of projective spaces $\mathbb{RP}^n\times \mathbb{RP}^{2m+1}$
	 (resp. $\mathbb{CP}^n\times\mathbb{CP}^{2m+1}$)  are also discussed.

\end{abstract}
\subjclass[2010]{Primary 55T10; Secondary 57S99 }

\keywords{Free action; Finitistic space; Leray-Serre spectral sequence; Gysin sequence}

\maketitle
\section {Introduction}
Let $G$ be a compact Lie group acting on a space $X$. An  interesting problem is to classify $X$ when $G$ acts freely on $X$ and the orbit space $X/G$ is known. First,  J. C. Su  \cite{Su1963} worked in this direction. He proved that if $G=\Z_2 $ or $\mathbb{S}^1$,  acts freely on a space $X$ and the orbit space $X/G$   is the mod 2 or integral cohomology projective space $\mathbb{FP}^n$, $\mathbb{F}=\mathbb{R}$ or $\mathbb{C}$, then the space $X$ is the mod 2 cohomology $n$-sphere $\mathbb{S}^{n}$ or integral cohomology $(2n+1)$-sphere $\mathbb{S}^{2n+1}$. Su also proved that if $G=\mathbb{Z}_q$, $q$ an odd prime, and $X/G=L_q^{2n+1}$ (lens space)  then $X$ is the  mod $q$ cohomology $(2n+1)$-sphere $\mathbb{S}^{2n+1}$. It has also been proved that if  $G=\S^3$ and $X/G=\mathbb{HP}^n$ (quaternion projective space), then $X$ is either integral cohomology $\S^{4n+3}$ or $\S^3 \times\mathbb{HP}^n$ \cite{Anju 2}. Harvey \cite{Harvey} proved that for semi-free actions of $G=\S^1$ when the  orbit space $X/G$ is  a   
simply connected 4-manifold, then the total space is a  simply connected 5-manifold.

The orbit spaces of free actions of $G=\Z_2$ on   projective spaces $X=\mathbb{FP}^{2n+1}$ and the product of   projective spaces $X=\mathbb{FP}^n\times\mathbb{FP}^{m}$, where $\mathbb{F}=\mathbb{R}$ or $\mathbb{C}$, have been  discussed in \cite{Hemant2008,Mahender2010}. For $G=\Z_2$ and $X=\mathbb{RP}^{2n+1}$ (resp.  $\mathbb{CP}^{2n+1}$), it is proved that the  mod 2 cohomology of the orbit space $X/G$ is $\mathbb{RP}^1\times \mathbb{CP}^n $  (resp. $ \mathbb{RP}^2\times \mathbb{HP}^n$) \cite{Hemant2008}. 
We have discussed  converse of these results. We have determined the cohomology classification  of a finitistic space $X$ if the mod 2 cohomology of orbit space $X/G$ is  $\mathbb{RP}^1\times \mathbb{CP}^n$ or  $ \mathbb{RP}^2\times \mathbb{HP}^n$.    
 For $G=\Z_2$ and $X=\mathbb{RP}^n\times\mathbb{RP}^{2m+1}$ (resp. $\mathbb{CP}^n\times\mathbb{CP}^{2m+1}$), one of the possibilities of the mod 2  cohomology of the orbit space is $\mathbb{RP}^1\times\mathbb{RP}^n\times  \mathbb{CP}^m$ (resp. $\mathbb{RP}^2 \times\mathbb{CP}^n\times \mathbb{HP}^m $) \cite{Mahender2010}. In this paper, we  have also  determined $X$ equipped with a free action of $G$ when the mod 2 cohomology of the orbit space  $X/G$ is $\mathbb{RP}^1\times\mathbb{RP}^n\times \mathbb{CP}^m $ or $\mathbb{RP}^2\times\mathbb{CP}^n\times \mathbb{HP}^m$. This gives a partial converse of results proved in \cite{Mahender2010}.

All spaces are assumed to finitistic connected spaces. Recall that a paracompact
Hausdorff space $X$ is said to be finitistic if every open cover of $X$ has a finite dimensional open refinement.
All compact spaces and all finite dimensional paracompact spaces are finitistic spaces. We have considered \v{C}ech cohomology with coefficients in $\Z_2$.
  Throughout the paper, $X\sim_2 Y$ means $H^*(X;\Z_2)\cong H^*(Y;\Z_2)$. 
We have proved the following results:
\begin{theorem}\label{theorem 0}
	Let $G=\Z_2$  act freely on a finitistic connected space $X$ with $X/G\sim_{\Z_2}\mathbb{RP}^1\times\mathbb{CP}^n$. Then  either  $X\sim_{\Z_2} \mathbb{S}^1\times\mathbb{CP}^n$ or $X\sim_{\Z_2}\mathbb{RP}^{2n+1}$.
\end{theorem}

\begin{theorem}\label{theorem 2}
	Let $G=\mathbb{Z}_2$ act freely on a finitistic connected space $X$  with $X/G\sim_{2}\mathbb{RP}^2\times \mathbb{HP}^n $. Then $X\sim_{2}\mathbb{S}^2\times \mathbb{HP}^n$ or $X\sim_{2}\mathbb{CP}^{2n+1}$.
\end{theorem}

\begin{theorem}\label{theorem 4}
	Let $G=\mathbb{Z}_2$ act freely on a finitistic connected space $X$  with $X/G\sim_2\mathbb{RP}^1\times\mathbb{RP}^n\times \mathbb{CP}^m $ and $\pi_1(B_G)$ act trivially on 
	$H^*(X)$. Then the cohomology algebra of  $X$  is one of the following: 
	\begin{enumerate}
		
		\item[(i)]  $\Z_2[x,y,z]/ \langle x^{n+1},y^{m+1},z^2+\alpha x^2+\beta y \rangle$,  where $\deg x=\deg z=1$, $\deg y=2$ and $\alpha,\beta\in\Z_2$;
		\item[(ii)] $\Z_2[x,y,z]/\langle x^2,y^{m+1},z^2+\alpha y^n\rangle$, where $\deg x=1,\deg y=2,\deg z=n$, $\alpha\in\Z_2$ and $\alpha=0$ if $n>m$.
	\end{enumerate}
\end{theorem}   
\begin{theorem}\label{theorem 5}
	Let $G=\mathbb{Z}_2$ act freely on a finitistic connected space $X$  with $X/G\sim_2\mathbb{RP}^2\times  \mathbb{CP}^n\times\mathbb{HP}^m$ and $\pi_1(B_G)$ acts trivially on 
	$H^*(X)$. Then the cohomology algebra of  $X$  is isomorphic to $$ \Z_2[x,y,z]/\langle x^{n+1},y^{m+1},z^2+\alpha x^2+\beta y\rangle,$$ where $\deg x=2,\deg y=4,\deg z=2$ and $\alpha,\beta \in \Z_2$.
\end{theorem}

\section{Preliminaries} 
In this section, we have recalled some results and definitions which are used to prove our main results.

Let $G$ be a compact Lie group and $G\hookrightarrow E_G\to B_G$ be the universal principal $G$-bundle, where $B_G$ is the classifying space. Suppose $G$  acts freely on a finitistic space $X$. The associated bundle $X\hookrightarrow (X\times E_G)/{G}\to B_G$ is a fibre bundle with fibre $X$. Put $X_G=(X\times E_G)/{G}$. Then the bundle $X\hookrightarrow X_G\to B_G$ is called the Borel fibration. We consider the Leray-Serre spectral sequence for the Borel fibration. If $\pi_1(B_G)$ acts trivially on $H^*(X)$, then  the system of local coefficients on $B_G$ is simple and the $E_2$-term of the Leray-Serre spectral sequence corresponding to the Borel fibration is given by \begin{equation*}
	E_2^{k,l}= H^k(B_G)\otimes H^l(X).
\end{equation*}
This sequence converges to $H^*(X_G)$ as an algebra.
We recall some results which are needed to prove our results:
\begin{proposition}\label{5p2}
	Let $X\stackrel{i}{\hookrightarrow} X_G \stackrel{\pi}{\longrightarrow} B_G$ be the Borel fibration. Suppose that the system of local coefficients on $B_G$ is simple, then the edge homomorphisms
	{\setlength\arraycolsep{35pt}
		\begin{align*}
		 H^k(B_G)\cong E_2^{k,0} \longrightarrow E_3^{k,0}\longrightarrow \cdots  
		   \longrightarrow E_k^{k,0} \longrightarrow E_{k+1}^{k,0}=E_{\infty}^{k,0}\subset H^k(X_G)
	\end{align*}}
	and  $$H^l(X_G) \longrightarrow E_{\infty}^{0,l}= E_{l+2}^{0,l} \subset E_{l+1}^{0,l} \subset \cdots \subset E_2^{0,l}\cong  H^l(X)$$
	are the homomorphisms $$\pi^*_k: H^k(B_G) \to H^k(X_G) ~ ~ ~ \textrm{and} ~ ~ ~ i^*_l: H^l(X_G)  \to H^l(X).$$
\end{proposition}
For details about spectral sequences, we refer \cite{McCleary}. Let $G$ act freely on $X$ and  $h:X_G\rightarrow X/G$ be the map induced by $G$-equivariant projection $X\times E_G\rightarrow X$.  Then $h$ is a homotopy equivalence \cite{Dieck}.  Now, we recall Gysin sequence of sphere bundles.
\begin{proposition}[\cite{Hatcher}]
	Let  $G=\Z_2$ act freely on a finitistic space $X$. The Gysin sequence of the sphere bundle $G\hookrightarrow X\stackrel{p}{\to} X/G$ is: \begin{align*}
	\cdots\longrightarrow H^{i-1}(X/G) \stackrel{\cup}{\longrightarrow} H^i(X/G)\stackrel{p^*_i}{\longrightarrow} H^i(X)\stackrel{\rho_i }{\longrightarrow}H^i(X/G)\stackrel{\cup}{\longrightarrow}H^{i+1}(X/G)\longrightarrow\cdots
	\end{align*}
	which  start with 
	\begin{align*}
		0\longrightarrow &H^{0}(X/G)\stackrel{p^*_0}{\longrightarrow}H^{0}(X)\stackrel{\rho_{0}}{\longrightarrow}H^0(X/G)\stackrel{\cup}{\longrightarrow}H^{1}(X/G)\stackrel{p^*_{1}}{\longrightarrow}H^{1}(X)\longrightarrow\cdots 	\end{align*}
 	where $\cup:H^{i}(X/G)\to H^{i+1}(X/G)$  maps $x\to xe$, $e\in H^{1}(X/G)$ denotes the image of $1$ under the map $\cup:H^0(X/G)\to H^1(X/G)$ and $e$ is called the Euler class of the   sphere  bundle. 
\end{proposition}

\begin{proposition}[\cite{Pergher}]\label{H^j(X/G)=0}
	Let  $G=\mathbb{Z}_2$ act freely on a finitistic space $X$. Suppose that $H^j(X)=0$ for all $j>n$, then $H^j(X/G)=0$ for all $j>n$.
\end{proposition}

\section{Proof of main Theorems} 
\begin{proof}[Proof of Theorem \ref{theorem 0}]	
	As $X/G\sim_2 \mathbb{RP}^1\times\mathbb{CP}^n $, we have $H^*(X/G)=\Z_2[a,b]/\langle a^2,b^{n+1}\rangle,$ where $ \deg a=1$ and $\deg b=2$. By  connectedness of $X$ and the exactness of Gysin sequence of $G$-bundle $G\hookrightarrow X\stackrel{p}{\to} X/G$, we get  
	$p^*_0$  is an isomorphism and the Euler class $\cup(1)$ is $a$. This implies that for  $0\leq i< n$, $p^*_{2i+2}:H^{2i+2}(X/G)\to H^{2i+2}(X)$,   $\rho_{2i+1}:H^{2i+1}(X)\to H^{2i+1}(X/G)$ and $\rho_{2n+1}:H^{2n+1}(X)\to H^{2n+1}(X/G)$ are isomorphisms.  Consequently, $H^{2i+1}(X)\cong H^{2n+1}(X)\cong H^{2i+2}(X)\cong\Z_2$ with bases $\{u_{2i+1}\}$, $\{u_{2n+1}\}$ and $\{x^{i+1}\}$, respectively, where $x=p^*_2(b)$, $\rho_{2i+1}(u_{2i+1})=ab^i$, $\rho_{2n+1}(u_{2n+1})=ab^n$. By Proposition \ref{H^j(X/G)=0}, $H^i(X)=0$ for $i>2n+1$. Now, we compute the cohomology algebra of $X$.  The Leray-Serre spectral sequence associated to the Borel fibration $X\hookrightarrow X_G\to B_G$  converges to $H^*(X_G)$. It is easy to see that $G$  acts trivially on $H^*(X)$. Thus, the $E_2$-term $E_2^{s,t}$ of the  Leray-Serre spectral sequence associated to the Borel fibration $X\stackrel{i}{\hookrightarrow}X_G\stackrel{\pi}{\to}B_G$ is   $H^s(B_G)\otimes H^t(X)$. By the edge homomorphism, we get $d_r(1\otimes x)=0$ for all $r> 0$ and $d_{r'}(1\otimes u_{1})\not=0$ for some $r'>0$. So, $d_2(1\otimes u_1)=t^2\otimes 1$.  This gives that $x^iu_1\not=0$, and hence $u_{2i+1}=x^iu_1$ for all $0\leq i\leq n$. Clearly, $x^{n+1}=0$. If $u_1^2=0$ then $X\sim_2\mathbb{S}^1\times \mathbb{CP}^n$.  If $u_1^2\not=0$ then $X\sim_2\mathbb{RP}^{2n+1}$.	
\end{proof}

\begin{proof}[Proof of Theorem \ref{theorem 2}]
	As $X/G\sim_2 \mathbb{RP}^2\times \mathbb{HP}^n $, we have $H^*(X/G)=\Z_2[a,b]/\langle a^3,b^{n+1}\rangle$, where $\deg a=1$ and $\deg b=4$.  It is easy to see that $\cup:H^0(X/G)\to H^1(X/G)$ maps 1 to $a$. Since $H^j(X/G)=0$ for $j\equiv 3\,(\text{mod 4})$, we get $p^*_{4i}:H^{4i}(X/G)\to H^{4i}(X)$ and $\rho_{4i+2}:H^{4i+2}(X)\to H^{4i+2}(X/G)$ are isomorphisms for $1\leq i\leq n$. Consequently, for $1\leq i\leq n$, we get  $H^{4i}(X)\cong H^{4i+2}(X)\cong\Z_2$  with bases $\{x^i\}$ and $\{u_{4i+2}\}$, respectively,  where $x=p_4^*(b)$ and $ \rho_{4i+2}(u_{4i+2})=a^2b^i$. By the exactness of Gysin sequence,  $H^{4i+3}(X)=H^{4i+1}(X)=0$ for all $0\leq i\leq n$.    Clearly, $H^{i}(X)=0$ for all $i> 4n+3$. Now, we compute cohomology algebra of $X$. As $\pi_1(B_G)$ acts trivially on $H^*(X)$, we get   $E^{s,t}_2=H^s(B_G)\otimes H^t(X)$. By the edge homomorphism, we get $d_3(1\otimes u_{2})=t^3\otimes 1$ and $d_3(1\otimes x)=0$. Thus, $x^iu_2\not=0$,  and hence $u_{4i+2}=x^iu_2$ for all $0\leq i\leq n$. Clearly, $x^{n+1}=0$. If $u_2^2=0$ then $X\sim_2\mathbb{S}^2\times \mathbb{HP}^n$. Otherwise, we get  $X\sim_2\mathbb{CP}^{2n+1}$.
\end{proof}

\begin{proof}[Proof of Theorem \ref{theorem 4}]
It is given that $X/G\sim_2\mathbb{RP}^1\times\mathbb{RP}^n\times \mathbb{CP}^m $. So, we get  $H^*(X/G)=\Z_2[a,b,c]/\langle a^{n+1},b^{m+1},c^2  \rangle$, where $\deg a=\deg c=1$ and $\deg b=2$. As $X$ is  connected,  $H^0(X)\cong \Z_2$. First, we consider $1\leq n<2m.$ By Proposition \ref{H^j(X/G)=0},  $H^i(X)=0$ for all $i>2m+n+1$.  Consider the Gysin sequence of the sphere bundle $G\hookrightarrow X\stackrel{p}{\to}X/G$:
\begin{align*}
		\cdots\longrightarrow H^{i-1}(X/G) \stackrel{\cup}{\longrightarrow} H^i(X/G)\stackrel{p^*_i}{\longrightarrow} H^i(X)\stackrel{\rho_i }{\longrightarrow}H^i(X/G)\stackrel{\cup}{\longrightarrow}H^{i+1}(X/G)\longrightarrow\cdots
\end{align*}
which begins with
\begin{align*}
	0\longrightarrow H^0(X/G)\stackrel{p^*_0}{\longrightarrow}H^0(X) \stackrel{\rho_0}{\longrightarrow}H^0(X/G)\stackrel{\cup}{\longrightarrow}H^1(X/G)
	\stackrel{p^*_1}{\longrightarrow} H^1(X)\longrightarrow\cdots
\end{align*}
 Clearly, the Euler class of the associated principal $G$-bundle $G\hookrightarrow X\to X/G$ must be nontrivial.  We have considered the following three cases:

\noindent\textbf{Case\,(1):} When $\cup:H^0(X/G)\to H^1(X/G)$ maps $1$ to $c$.

 For $0\leq i< \frac{n}{2}$, we have $H^{2i}(X/G) \cong \Z_2\oplus\cdots\oplus \Z_2$ ($2i+1$-copies) and $ H^{2i+1}(X/G)\cong \Z_2\oplus\cdots \oplus\Z_2$ ($2i+2$-copies) with bases
\begin{align*}
&\{b^i,ab^{i-1}c,a^2b^{i-1},a^3b^{i-2}c,\cdots, a^{2i-1}c,a^{2i}\}, \text{ and }\\
&\{b^ic,ab^{i},a^2b^{i-1}c,a^3b^{i-1},\cdots, a^{2i}c,a^{2i+1}\},
\end{align*} respectively.
 By the  exactness of Gysin sequence, we get  $\ker p_{2i}^*=\im \rho_{2i}
 \cong\Z_2\oplus\cdots\oplus\Z_2$ ($i$-copies) with basis $\{a^{2j+1}b^{i-j-1}c| 0\leq j\leq i-1\}$, and $\ker p_{2i+1}^*=\im \rho_{2i+1}
 \cong\Z_2\oplus\cdots\oplus\Z_2$ ($i+1$-copies) with basis $\{a^{2j}b^{i-j}c| 0\leq j\leq i\}$. Consequently, $\im p ^*_{2i}\cong \Z_2\oplus\cdots\oplus\Z_2$ ($i+1$-copies) with basis $\{x^{2j}y^{i-j}|0\leq j\leq i\}$, and $\im p ^*_{2i+1}\cong \Z_2\oplus\cdots\oplus\Z_2$ ($i+1$-copies) with basis $\{x^{2j+1}y^{i-j}|0\leq j\leq i\}$, where $p^*_1(a)=x,p^*_2(b)=y$.
 It is easy to see that 
 \begin{align*}
 	 0\to \im p^*_{j}\hookrightarrow H^{j}(X)\stackrel{\rho_{j}}{\to} \im \rho_{j}\to 0, j=2i,2i+1 
  \end{align*}
 are split exact sequences. This gives that   \begin{align*}
 	&H^{2i}(X)\cong  \Z_2\oplus\cdots\oplus\Z_2\, (2i+1\text{-copies}), \text{ and}\\
&H^{2i+1}(X)\cong  \Z_2\oplus\cdots\oplus\Z_2\,(2i+2\text{-copies})
\end{align*} with  bases $\{x^{2j}y^{i-j},u_{2i}^{2l+1}|0\leq j\leq i, 0\leq l\leq i-1\}$ and  $\{x^{2j+1}y^{i-j},u_{2i+1}^{2j}|0\leq j\leq i\}$, respectively, where  $\rho_{2i}(u_{2i}^{2l+1})=a^{2l+1}b^{i-l-1}c, 0\leq l\leq i-1\text{ and }\rho_{2i+1}(u_{2i+1}^{2j})=a^{2j}b^{i-j}c, 0\leq j\leq i$.

  Firstly, suppose that $n=2k$ for some $k$. Then for $\frac{n}{2}\leq  i\leq  m$,  $H^{2i}(X/G) \cong H^{2i+1}(X/G)\cong \Z_2\oplus\cdots\oplus \Z_2$ ($n+1$-copies) with bases
  \begin{align*} &\{b^i,ab^{i-1}c,a^2b^{i-1},a^3b^{i-2}c,\cdots, a^{n-1}b^{i-k}c,a^{n}b^{i-k}\}, \text{ and }\\
&\{b^ic,ab^{i},a^2b^{i-1}c,a^3b^{i-1},\cdots, a^{n-1}b^{i-k+1},a^{n}b^{i-k}c\},
\end{align*} respectively.
 Again, by the exactness of  Gysin sequence,  $\ker p_{2i}^*=\im\rho_{2i}\cong \Z_2\oplus\cdots\oplus \Z_2$ ($k$-copies) and $\ker p^*_{2i+1}=\im \rho_{2i+1}\cong \Z_2\oplus\cdots\oplus \Z_2$ ($k+1$-copies) with bases $\{a^{2j+1}b^{i-j-1}c|0\leq j\leq k-1\}$ and $\{a^{2j}b^{i-j}c|0\leq j\leq k\}$, respectively.  This gives that $\im p^*_{2i}\cong \Z_2\oplus\cdots\oplus \Z_2$ ($k+1$-copies) and $\im p^*_{2i+1}\cong \Z_2\oplus\cdots\oplus \Z_2$ ($k$-copies) with bases $\{x^{2j}y^{i-j}|0\leq j\leq k\}$ and $\{x^{2j+1}y^{i-j}|0\leq j\leq k-1\}$, respectively. We observe that $H^j(X)\cong \im p^*_j\oplus \im\rho_j, j=2i,2i+1$.
  So, we get
  \begin{align*}
  H^{2i}(X)\cong H^{2i+1}(X)\cong \Z_2\oplus\cdots\oplus \Z_2\,(n+1\text{-copies}) 
  \end{align*} 
  with bases $\{x^{2j}y^{i-j},u_{2i}^{2l+1}|0\leq j\leq k,0\leq l\leq k-1\}$ and $\{x^{2j+1}y^{i-j},u_{2i+1}^{2l}|0\leq j\leq k-1,0\leq l\leq k\}$, respectively, where $\rho_{2i}(u_{2i}^{2l+1})=a^{2l+1}b^{i-l-1}c,0\leq l\leq k-1  \text{ and }\rho_{2i+1}(u_{2i+1}^{2l})=a^{2l}b^{i-l}c,0\leq l\leq k$.

 For $m<  i\leq m+\frac{n}{2}$, we have $H^{2i}(X/G) \cong  \Z_2\oplus\cdots\oplus \Z_2$ ($n+2m-2i+2$-copies) and $H^{2i+1}(X/G)\cong \Z_2\oplus\cdots\oplus \Z_2$ ($n+2m-2i+1$-copies) with bases 
 \begin{align*}
 &\{a^{2i-2m-1}b^mc,a^{2i-2m}b^{m},a^{2i-2m+1}b^{m-1}c,a^{2i-2m+2}b^{m-1},\cdots,  a^{n-1}b^{i-k}c,a^{n}b^{i-k}\}, \text{ and }\\ 
 &\{a^{2i-2m}b^mc,a^{2i-2m+1}b^{m},a^{2i-2m+2}b^{m-1}c,a^{2i-2m+3}b^{m-1},\cdots, a^{n-1}b^{i-k+1},a^{n}b^{i-k}c\}, \end{align*} respectively. 
 Note that  $\ker p_{2i}^*=\im\rho_{2i}\cong \Z_2\oplus\cdots\oplus \Z_2$ ($k+m-i+1$-copies) and $\ker p^*_{2i+1}=\im \rho_{2i+1}\cong \Z_2\oplus\cdots\oplus \Z_2$ ($k+m-i+1$-copies) with bases $\{a^{2j+1}b^{i-j-1}c|i-m-1\leq j\leq k-1\}$ and $\{a^{2j}b^{i-j}c|i-m\leq j\leq k\}$, respectively.  This implies that $\im p^*_{2i}\cong \Z_2\oplus\cdots\oplus \Z_2$ ($k+m-i+1$-copies) and $\im p^*_{2i+1}\cong \Z_2\oplus\cdots\oplus \Z_2$ ($k+m-i$-copies) with bases $\{x^{2j}y^{i-j}|i-m\leq j\leq k\}$ and $\{x^{2j+1}y^{i-j}|i-m\leq j\leq k-1\}$, respectively. 
So, we get
 \begin{align*}
 	&H^{2i}(X)\cong  \Z_2\oplus\cdots\oplus \Z_2\,(n+2m-2i+2\text{-copies}) \text{, and}\\
 		&H^{2i+1}(X)\cong  \Z_2\oplus\cdots\oplus \Z_2\,(n+2m-2i+1\text{-copies})
 \end{align*} 
 with bases $\{x^{2j}y^{i-j},u_{2i}^{2l+1}|i-m\leq j\leq k,i-m-1\leq l\leq k-1\}$ and $\{x^{2j+1}y^{i-j},u_{2i+1}^{2l}|i-m\leq j\leq k-1,i-m\leq l\leq k\}$, respectively, where $\rho_{2i}(u_{2i}^{2l+1})=a^{2l+1}b^{i-l-1}c, i-m-1\leq l\leq k-1$ and $\rho_{2i+1}(u_{2i+1}^{2l})=a^{2l}b^{i-l}c,i-m\leq l\leq k$.

 Now, we compute the cohomology algebra of $X$. 
Consider the Leray Serre spectral sequence for the Borel fibration $X\stackrel{i}{\hookrightarrow}X_G\stackrel{\pi}{\rightarrow}B_G$ with $E_2^{s,t}=H^s(B_G)\otimes H^t(X)$.  
As $p^*_j=i^*_j\circ h^*$ for all $j\geq 0$, we get $d_r(1\otimes x)=0=d_r(1\otimes y)$ for all $r\geq 0$ and for all $i,j$;  $d_{r_{i,j}}(1\otimes u_j^i)\not=0$ for some $r_{i,j}>0$. Put $u_1^0=z$. So, $H^1(X)$ is generated by $\{x,z\}$. Trivially, $d_2(1\otimes z)=t^2\otimes 1$ and so $d_2(1\otimes xz)\not=0$. This implies that
$xz=\alpha x^2+\beta y+u_2^1,\alpha,\beta\in \Z_2$. So, 
 $H^2(X)$  is generated by $\{x^2,y,xz\}$.
 By induction,  we show that for $1\leq i<\frac{n}{2}$, $H^{2i-1}(X)$ and $H^{2i}(X)$ are generated by $\{xy^{i-1},x^3y^{i-2},\cdots,x^{2i-3}y,x^{2i-1},y^{i-1}z,x^2y^{i-2}z,\cdots x^{2i-4}yz,x^{2i-2}z\}$ and $\{y^i,x^2y^{i-1},\cdots, x^{2i-2}y,x^{2i},xy^{i-1}z,x^3y^{i-2}z,\cdots,x^{2i-1}z\}$, respectively. Now, we prove it for $i+1$. 
 We have $d_2(1\otimes y^iz)=t^2\otimes y^i, d_2(1\otimes x^2y^{i-1}z)=t^2\otimes x^2y^{i-1},\cdots, d_2(1\otimes x^{2i-2}yz)=t^2\otimes x^{2i-2}y, d_2(1\otimes x^{2i}z)=t^2\otimes x^{2i}.$ This gives that $y^iz,x^2y^{i-1}z,\cdots x^{2i-2}yz,x^{2i}z$ are nonzero elements of $H^{2i+1}(X)$. Therefore, 
 \begin{align*}
 	y^iz=&\alpha_1^0 xy^{i}+\alpha_3^0 x^3 y^{i-1}+\cdots+\alpha_{2i-1}^0 x^{2i-1} y+\alpha_{2i+1}^0 x^{2i+1}\\
 	&+\beta_0^0u^0_{2i+1}+\beta_2^0u^2_{2i+1}+\cdots+\beta^0_{2i-2}u_{2i+1}^{2i-2}+\beta_{2i}^0u^{2i}_{2i+1}\\
 	x^2y^{i-1}z=&\alpha_1^2 xy^{i}+\alpha_3^2 x^3 y^{i-1}+\cdots+\alpha_{2i-1}^2 x^{2i-1} y+\alpha_{2i+1}^2 x^{2i+1}\\
 	&+\beta_0^2u^0_{2i+1}+\beta_2^2u^2_{2i+1}+\cdots+\beta^2_{2i-2}u_{2i+1}^{2i-2}+\beta_{2i}^2u^{2i}_{2i+1}\\
 	\vdots\\
 	x^{2i}z=&\alpha_1^{2i} xy^{i}+\alpha_3^{2i} x^3 y^{i-1}+\cdots+\alpha_{2i-1}^{2i} x^{2i-1} y+\alpha_{2i+1}^{2i} x^{2i+1}\\
 		&+\beta_0^{2i}u^0_{2i+1}+\beta_2^{2i}u^2_{2i+1}+\cdots+\beta^{2i}_{2i-2}u_{2i+1}^{2i-2}+\beta_{2i}^{2i}u^{2i}_{2i+1}
\end{align*}
 where $\beta_j^k\not=0$ for some $ j\in\{0,2,\cdots, 2i\}$ and for all $k=0,2,\cdots, 2i$.  
  From this we get

 \begin{align*}
 	t^2\otimes	y^i=&\beta_0^0d_2(1\otimes u^0_{2i+1})+\beta_2^0d_2(1\otimes u^2_{2i+1})+\cdots+\beta_{2i}^0d_2(1\otimes u_{2i+1}^{2i})\\
 	t^2\otimes	x^2y^{i-1}=
 	&\beta_0^2d_2(1\otimes u^0_{2i+1})+\beta_2^2d_2(1\otimes u^2_{2i+1})+\cdots+\beta_{2i}^2d_2(1\otimes u^{2i}_{2i+1})\\
 	\vdots\\
 	 	t^2\otimes	x^{2i}=
 	&\beta_0^{2i}d_2(1\otimes u^0_{2i+1})+\beta_2^{2i}d_2(1\otimes u^2_{2i+1})+\cdots+\beta_{2i}^{2i}d_2(1\otimes u^{2i}_{2i+1})
 \end{align*}
 As the set $\{y^i,x^2y^{i-1},\cdots, x^{2i-2}y,x^{2i}\}$ is linearly independent,  the coefficient matrix
 $$
 \begin{bmatrix}
 	\beta_0^0 & \beta_2^0 &\cdots &\beta_{2i}^0\\
 		\beta_0^2 & \beta_2^2 &\cdots &\beta_{2i}^2\\
 		\vdots &\vdots &\vdots &\vdots \\
 			\beta_0^{2i} & \beta_2^{2i} &\cdots &\beta_{2i}^{2i}
 \end{bmatrix}
 $$
 is invertible. Then  $u^{2k}_{2i+1}$ is generated by  $\{xy^i,x^3y^{i-1},\cdots,x^{2i-1}y,x^{2i+1},y^iz,x^2y^{i-1}z,\cdots,$\\$x^{2i-2}z,x^{2i}z\}$ for all  $0\leq k\leq i$. So, the basis for $H^{2i+1}(X)$ is $\{xy^i,x^3y^{i-1},\cdots, x^{2i-3}y^2,x^{2i-1}y$\\$,x^{2i+1},y^iz,x^2y^{i-1}z,\cdots,x^{2i-2}yz,x^{2i}z\}$. 
  Also, we have $d_2(1\otimes xy^iz)=t^2\otimes xy^i,d_2(1\otimes x^3y^{i-1}z)=t^2\otimes x^3y^{i-1},d_2(1\otimes x^5y^{i-2}z)=t^2\otimes x^5y^{i-2},\cdots,d_2(1\otimes x^{2i+1}z)=t^2\otimes x^{2i+1}$. This implies that $xy^iz,x^3y^{i-1}z,\cdots,x^{2i+1}z$ are nonzero elements of $H^{2i+2}(X)$. So, we get
\begin{align*}
	xy^iz=&\alpha_0^1 y^{i+1}+\alpha_2^1 x^2 y^i+\cdots+\alpha_{2i-2}^1 x^{2i-2} y^2+\alpha_{2i}^1 x^{2i} y\\
	&+\beta_1^1u^1_{2i+2}+\beta_3^1u^3_{2i+2}+\cdots+\beta_{2i+1}^1u^{2i+1}_{2i+2}	\\
		x^3y^{i-1}z=&\alpha_0^3 y^{i+1}+\alpha_2^3 x^2 y^i+\cdots+\alpha_{2i-2}^3 x^{2i-2} y^2+\alpha_{2i}^3 x^{2i} y\\
		&+\beta_1^3u^1_{2i+2}+\beta_3^3u^3_{2i+2}+\cdots+\beta_{2i+1}^3u^{2i+1}_{2i+2}
			\end{align*}
		\begin{align*}
		\vdots\\
			x^{2i+1}z=&\alpha_0^{2i+1} y^{i+1}+\alpha_2^{2i+1}x^2 y^i+\cdots+\alpha_{2i-2}^{2i+1} x^{2i-2} y^2+\alpha_{2i}^{2i+1} x^{2i} y\\
			&+\beta_1^{2i+1}u^1_{2i+2}+\beta_3^{2i+1}u^3_{2i+2}+\cdots+\beta_{2i+1}^{2i+1}u^{2i+1}_{2i+2}
\end{align*}
where $\beta_j^k\not=0$ for some $ j\in \{1,3,\cdots, 2i+1\}$ and for all $k=1,3,\cdots, 2i+1$. Thus, we get
\begin{align*}
t^2\otimes	xy^i=&\beta_1^1d_2(1\otimes u^1_{2i+2})+\beta_3^1d_2(1\otimes u^3_{2i+2})+\cdots+\beta_{2i+1}^1d_2(1\otimes u^{2i+1}_{2i+2})\\
t^2\otimes	x^3y^{i-1}=
	&\beta_1^3d_2(1\otimes u^1_{2i+2})+\beta_3^3d_2(1\otimes u^3_{2i+2})+\cdots+\beta_{2i+1}^3d_2(1\otimes u^{2i+1}_{2i+2})\\
	\vdots\\
t^2\otimes	x^{2i+1}=
	&\beta_1^{2i+1}d_2(1\otimes u^1_{2i+2})+\beta_3^{2i+1}d_2(1\otimes u^3_{2i+2})+\cdots+\beta_{2i+1}^{2i+1}d_2(1\otimes u^{2i+1}_{2i+2})
\end{align*}
As the set $ \{xy^i,x^3y^{i-1},\cdots, x^{2i+1}\}$  is linearly independent, $u^k_{2i+2}$
 is generated by $\{y^{i+1},x^2y^{i},$\\$\cdots, x^{2i}y,xy^{i}z,x^3y^{i-1}z,\cdots, x^{2i+1}z\}$ for all $k=1,3,\cdots, 2i+1$. So, our assumption is true for $i+1$.  Similarly,  for $\frac{n}{2}\leq i\leq m$, bases for $H^{2i}(X)$ and $H^{2i+1}(X)$ are  $\{x^{2j}y^{i-j},x^{2l+1}y^{i-l-1}z$\\$|0\leq j\leq k,0\leq l\leq k-1\}$ and $\{x^{2j+1}y^{i-j},x^{2l}y^{i-l}z|0\leq j\leq k-1,0\leq l\leq k\}$, respectively, and for  $m< i\leq m+\frac{n}{2}$, bases for $H^{2i}(X)$ and $H^{2i+1}(X)$ are $\{x^{2j}y^{i-j},x^{2l+1}y^{i-l-1}z|i-m\leq j\leq k,i-m-1\leq l\leq k-1\}$ and $\{x^{2j+1}y^{i-j},x^{2l}y^{i-l}z|i-m\leq j\leq k-1,i-m\leq l\leq k\}$, respectively.

Now, suppose that  $n=2k+1$ for some $k$. For $n\leq 2i,2i+1\leq 2m$,  $ H^{n-1}(X/G)\cong \Z_2\oplus\cdots\oplus \Z_2$ ($2k+1$-copies) and  $H^{2i}(X/G) \cong H^{2i+1}(X/G)\cong \Z_2\oplus\cdots\oplus \Z_2$ ($2k+2$-copies) with bases 
\begin{align*}
	&\{b^k,ab^{k-1}c,a^2b^{k-1},a^3b^{k-2}c,\cdots,a^{2k-1}c,a^{2k}\},\\
	&\{b^i,ab^{i-1}c,a^2b^{i-1},a^3b^{i-2}c,\cdots,a^{n-1}b^{i-k},a^{n}b^{i-k-1}c\}, \text{ and }\\
&\{b^ic,ab^{i},a^2b^{i-1}c,a^3b^{i-1},\cdots, a^{n-1}b^{i-k}c,a^{n}b^{i-k}\},
\end{align*}
 respectively. Then we have $\ker p^*_{n-1}=\im \rho_{n-1}\cong \Z_2\oplus\cdots\oplus\Z_2$ ($k$-copies),
$\ker p^*_{2i}=\im \rho_{2i}\cong \Z_2\oplus\cdots\oplus\Z_2$ ($k+1$-copies) and
$\ker p^*_{2i+1}=\im \rho_{2i+1}\cong \Z_2\oplus\cdots\oplus\Z_2$ ($k+1$-copies) 
  with bases $\{a^{2j+1}b^{k-j-1}c|0\leq j\leq k-1\}$, $\{a^{2j+1}b^{i-j-1}c|0\leq j\leq k\}$ and $\{a^{2j}b^{i-j}c|0\leq j\leq k\}$, respectively. So, $\im p^*_{n-1}\cong\Z_2\oplus\cdots\oplus\Z_2$ ($k+1$-copies), $\im p^*_{2i}\cong\Z_2\oplus\cdots\oplus\Z_2$ ($k+1$ copies) and  $\im p^*_{2i+1}\cong\Z_2\oplus\cdots\oplus\Z_2$ ($k+1$ copies) with bases $\{x^{2j}y^{k-j}|0\leq j\leq k\}$,  $\{x^{2j}y^{i-j}|0\leq j\leq k\}$ and $\{x^{2j+1}y^{i-j}|0\leq i\le k\}$, respectively.
As we have done above, we get 
  \begin{align*}
  &	H^{n-1}(X)\cong \Z_2\oplus\cdots\oplus \Z_2\,(n\text{-copies}), \text{ and}
  	\end{align*}
  \begin{align*}
&	H^{2i}(X)\cong H^{2i+1}(X)\cong \Z_2\oplus\cdots\oplus \Z_2\,(n+1\text{-copies}) 
\end{align*} 
with bases $\{x^{2j}y^{k-j},u_{2k}^{2l+1}|0\leq j\leq k,0\leq l\leq k-1\}$, $\{x^{2j}y^{i-j},u_{2i}^{2j+1}|0\leq j\leq k\}$ and $\{x^{2j+1}y^{i-j},u_{2i+1}^{2j}|0\leq j\leq k\}$, respectively, where $\rho_{2k}(u_{2k}^{2l+1})=a^{2l+1}b^{k-l-1}c,0\leq l\leq k-1$, $\rho_{2i}(u_{2i}^{2j+1})=a^{2j+1}b^{i-j-1}c, 0\leq j\leq k$  and $\rho_{2i+1}(u_{2i+1}^{2j})=a^{2j}b^{i-j}c,0\leq j\leq k$.
Also, for  
$2m< 2i,2i+1\leq 2m+n+1$, we have
    $H^{2i}(X/G) \cong  \Z_2\oplus\cdots\oplus \Z_2$ ($n+2m-2i+2$-copies) and $ H^{2i+1}(X/G)\cong \Z_2\oplus\cdots\oplus \Z_2$ ($n+2m-2i+1$-copies) with bases
    \begin{align*} &\{a^{2i-2m-1}b^mc,a^{2i-2m}b^{m},a^{2i-2m+1}b^{m-1}c,a^{2i-2m+2}b^{m-1},\cdots ,  a^{n-1}b^{i-k},a^{n}b^{i-k-1}c\}, \text{ and }\\ 
    &\{a^{2i-2m}b^mc,a^{2i-2m+1}b^{m},a^{2i-2m+2}b^{m-1}c,a^{2i-2m+3}b^{m-1},\cdots, a^{n-1}b^{i-k}c,a^{n}b^{i-k}\},
 \end{align*} respectively. 
    Note that  $\ker p_{2i}^*=\im\rho_{2i}\cong \Z_2\oplus\cdots\oplus \Z_2$ ($k+m-i+2$-copies) and $\ker p^*_{2i+1}=\im \rho_{2i+1}\cong \Z_2\oplus\cdots\oplus \Z_2$ ($k+m-i+1$-copies) with bases $\{a^{2j+1}b^{i-j-1}c|i-m-1\leq j\leq k\}$ and $\{a^{2j}b^{i-j}c|i-m\leq j\leq k\}$, respectively.  This implies that $\im p^*_{2i}\cong \Z_2\oplus\cdots\oplus \Z_2$ ($k+m-i+1$-copies) and $\im p^*_{2i+1}\cong \Z_2\oplus\cdots\oplus \Z_2$ ($k+m-i+1$-copies) with bases $\{x^{2j}y^{i-j}|i-m\leq j\leq k\}$ and $\{x^{2j+1}y^{i-j}|i-m\leq j\leq k\}$, respectively. 
    So, we get
    \begin{align*}
    	&H^{2i}(X)\cong  \Z_2\oplus\cdots\oplus \Z_2\,(n+2m-2i+2\text{-copies}) \text{, and}\\
    	  	&H^{2i+1}(X)\cong  \Z_2\oplus\cdots\oplus \Z_2\,(n+2m-2i+1\text{-copies})
    \end{align*} 
    with bases $\{x^{2j}y^{i-j},u_{2i}^{2l+1}|i-m\leq j\leq k,i-m-1\leq l\leq k\}$ and $\{x^{2j+1}y^{i-j},u_{2i+1}^{2j}|i-m\leq$\\$ j\leq k\}$, respectively, where $\rho_{2i}(u_{2i}^{2l+1})=a^{2l+1}b^{i-l-1}c, i-m-1\leq l\leq k$ and $\rho_{2i+1}(u_{2i+1}^{2j})=a^{2j}b^{i-j}c,i-m\leq j\leq k$.

   For $1\leq i<n$, $H^i(X)$ has the same basis as in the case when $n$ is even. Again by induction,
 for $n\leq 2i,2i+1\leq 2m$,  bases for $H^{n-1}(X)$, $H^{2i}(X)$ and $H^{2i+1}(X)$ are $\{x^{2j}y^{k-j},x^{2l+1}y^{k-l-1}z|0\leq j\leq k,0\leq l\leq k-1\}$, $\{x^{2j}y^{i-j},x^{2l+1}y^{i-l-1}z|0\leq j,l\leq k\}$ and $\{x^{2j+1}y^{i-j},x^{2l}y^{i-l}z|0\leq j,l\leq k\}$, respectively, and for  $m< i\leq m+\frac{n}{2}$, bases for $H^{2i}(X)$ and $H^{2i+1}(X)$ are $\{x^{2j}y^{i-j},x^{2l+1}y^{i-l-1}z|i-m\leq j\leq k,i-m-1\leq l\leq k\}$ and $\{x^{2j+1}y^{i-j},x^{2l}y^{i-l}z|i-m\leq j,l\leq k\}$, respectively.

 As $d_2(1\otimes z^2)=0$, we get $z^2=\alpha x^2+\beta y$ for some $\alpha,\beta \in \Z_2$.
Thus the cohomology algebra of $X$ is given by $$\Z_2[x,y,z]/\langle x^{n+1},y^{m+1},z^2+\alpha x^2+\beta y\rangle,$$ where  $\deg x=\deg z=1$, $\deg y=2$ and $\alpha,\beta \in \Z_2$.
    This realizes possibility (i).

\noindent\textbf{Case\,(2):} When $\cup:H^0(X/G)\to H^1(X/G)$ maps 1 to  $a$.

\noindent Clearly, $\rho_i:H^i(X)\to H^i(X/G)$  is trivial for all $1\leq i< n$. For $0<2i,2i+1<n$, bases for $\ker p^*_{2i}$ and $\ker p^*_{2i+1}$ are $\{a^{2j+1}b^{i-j-1}c,a^{2l}b^{i-l}|0\leq j\leq i-1,1\leq l\leq i\}$ and $\{a^{2j+1}b^{i-j},a^{2l}b^{i-l}c|0\leq j\leq i,1\leq l\leq i\}$, respectively.  This implies that $\im p^*_{2i}\cong\im p^*_{2i+1}\cong \Z_2$ with bases $\{y^i\}$ and $\{xy^i\}$, respectively, where $p_1^*(c)=x$ and $p_2^*(b)=y$. So, $H^{2i}(X)\cong  H^{2i+1}(X)\cong \Z_2$ with  bases $\{y^i\}$ and $\{xy^i\}$, respectively.

 If $n=2k$ for some $k$, then for $n\leq 2i,2i+1\leq 2m+1$, 
$\{a^{2j+1}b^{i-j-1}c,a^{2l}b^{i-l}|0\leq j\leq k-1,1\leq l\leq k\}$ and $\{a^{2j+1}b^{i-j},a^{2l}b^{i-l}c|0\leq j\leq k-1,1\leq l\leq k\}$  are bases for $\ker p^*_{2i}$ and $\ker p^*_{2i+1}$, respectively. This gives that  $\im p^*_{2i}\cong \im p^*_{2i+1}\cong \Z_2$ with bases $\{y^i\}$ and $\{xy^i\}$, respectively, and $\im \rho_{2i}\cong\im\rho_{2i+1}\cong\Z_2$ with bases $\{a^nb^{i-k}\}$ and $\{a^nb^{i-k}c\}$, respectively.  Consequently, $H^{2i}(X)\cong H^{2i+1}(X)\cong\Z_2\oplus\Z_2$ with bases $\{y^i,u^n_{2i}\}$ and $\{xy^i,u^n_{2i+1}\}$, respectively, where $\rho_{2i}(u^n_{2i})=a^nb^{i-k}$ and $\rho_{2i+1}(u^n_{2i+1})=a^nb^{i-k}c$. If $2m+1<2i,2i+1\leq 2m+n+1$ then bases for $\ker p^*_{2i}$ and $\ker p^*_{2i+1}$ are 
$\{a^{2j+1}b^{i-j-1}c,a^{2l}b^{i-l}|i-m-1\leq j\leq k-1,i-m\leq l\leq k\}$ and $\{a^{2j+1}b^{i-j},a^{2l}b^{i-l}c|i-m\leq j\leq k-1,i-m\leq l\leq k\}$, respectively. This implies that $p^*_{2i},p^*_{2i+1}$ are trivial and $\im \rho_{2i}\cong\im\rho_{2i+1}\cong\Z_2$ with bases $\{a^nb^{i-k}\}$ and $\{a^nb^{i-k}c\}$, respectively. Thus, $H^{2i}(X)\cong H^{2i+1}(X)\cong\Z_2$ with bases $\{u_{2i}^n\}$ and $\{u^{n}_{2i+1}\}$ respectively, where $\rho_{2i}(u^n_{2i})=a^nb^{i-k}$ and $\rho_{2i+1}(u^n_{2i+1})=a^nb^{i-k}c$. Trivially, $x^2=y^{m+1}=0$.
Put $u_{2k}^n=z$.  Then by the edge homomorphism, $d_r(1\otimes x)=d_r(1\otimes y)=0$ for all $r>0$ and $d_{r'}(1\otimes z)\not =0$ for some $r'>0$. As $X/G\sim_{2}\mathbb{RP}^1\times\mathbb{RP}^n\times\mathbb{CP}^m$, $r'$ must be $n+1$. Note that  for all $0\leq i\leq m$,  $xy^iz$ and $xy^i$ are nonzero elements of $H^{n+2i+1}(X)$ and $H^{n+2i}(X)$, respectively. Therefore, $xy^iz=\alpha xy^{k+m}+\beta u^n_{n+2i+1}$ and $xy^i=\alpha' y^{k+m}+\beta' u^n_{n+2i}$ for some $\alpha,\alpha',\beta,\beta'\in \Z_2$, where $\beta$ and  $\beta'$ are nonzero. Note that for $n+2i>2m$, we get  $\alpha=\alpha'=0$.

As 
$d_r(1\otimes z^2)=0$ for all $r\geq 0$, we get $z^2=\alpha y^n$ for some  $\alpha\in \Z_2$ and $\alpha=0$ for $n>m$. Thus, the cohomology algebra of $X$ is given by $$\Z_2[x,y,z]/\langle x^2,y^{m+1},z^2+\alpha y^n\rangle,$$ where $\deg x=1,\deg y=2,\deg z=n$ and $\alpha=0$ if $n>m$. This realizes possibility (ii).

 Now,  let $n=2k+1$ for some $k\geq 0$.  For $n\leq 2i,2i+1\leq 2m+1$, we get $H^{2i}(X)\cong \Z_2\oplus\Z_2\cong H^{2i+1}(X)$ with bases $\{y^i,u^n_{2i}\}$ and $\{xy^i,u^n_{2i+1}\}$, respectively, where $\rho_{2i}(u^n_{2i})=a^nb^{i-k-1}c$ and $ \rho_{2i+1}(u^n_{2i+1})=a^nb^{i-k}$. Also, for $2m+1<2i,2i+1\leq 2m+n+1$,  we have $H^{2i}(X)\cong \Z_2\cong H^{2i+1}(X)$ with bases $\{u^n_{2i}\}$ and $\{u^n_{2i+1}\}$, respectively, where $\rho_{2i}(u^n_{2i})=a^nb^{i-k-1}c$ and $\rho_{2i+1}(u^n_{2i+1})=a^nb^{i-k}$. It is easy to observe that $X$  has the same cohomology algebra as in the case when $n$ is even.

 \noindent\textbf{Case\,(3):} When $\cup:H^0(X/G)\to H^4(X/G)$ maps $1$ to $a+c$. 
 
 Also, in this case $\rho_i:H^i(X)\to H^i(X/G)$ and $p^*_j$ are trivial for all $1\leq i<n$ and $2m+1<j\leq 2m+n+1$. Similar calculations show that for $1\leq i<n$,   $H^{2i}(X)\cong H^{2i+1}(X)\cong \Z_2$ with bases $\{y^i\}$ and $\{xy^i\}$, respectively, where $x=p^*_1(a)=p_1^*(c),y=p_2^*(b)$. For $n\leq 2i,2i+1\leq 2m+1$,  $H^{2i}(X)\cong H^{2i+1}(X)\cong \Z_2\oplus\Z_2$ with bases $\{y^i,u_{2i}\}$, $\{xy^i,u_{2i+1}\}$, respectively; and for  $2m+1<2i,2i+1\leq 2m+n+1$, we have $H^{2i}(X)\cong H^{2i+1}(X)\cong \Z_2$ with bases $\{u_{2i}\}$, $\{u_{2i+1}\}$, respectively, where 
 $\rho_{2i}(u_{2i})=a^{n-1}b^{i-k}c+a^nb^{i-k}$ and $\rho_{2i+1}(u_{2i+1})=a^nb^{i-k}c$ if $n=2k$; and $\rho_{2i}(u_{2i})=a^{n}b^{i-k-1}c$ and $\rho_{2i+1}(u_{2i+1})=a^{n-1}b^{i-k}c+a^nb^{i-k}$ if $n=2k+1$. 
 Consequently, the cohomology algebra of $X$ is given by  $$ \Z_2[x,y,z]/\langle x^2,y^{m+1},z^2+\alpha y^n\rangle,$$ where $\deg x=1,\deg y=2,\deg z=n$ and $\alpha=0$ if $n>m$. This realizes possibility (ii).
 
Similarly, we get the same cohomology algebra   for  the case when $2m\leq n$.\qedhere
  \end{proof}

\begin{proof}[Proof of Theorem \ref{theorem 5}]
As $X/G\sim_2\mathbb{CP}^n\times \mathbb{HP}^m\times \mathbb{RP}^2$, we get $H^*(X/G)=\Z_2[a,b,c]/\langle a^{n+1}$\\$,b^{m+1},c^3  \rangle$, where $\deg a=2,\deg b=4$ and $\deg c=1$. Since $X$ is  connected,  $H^0(X)\cong \Z_2$. By Proposition \ref{H^j(X/G)=0},  $H^i(X)=0$ for all $i>4m+2n+2$. Firsty, assume that $1\leq n<2m$.
	Consider the Gysin sequence of the sphere bundle $G\hookrightarrow X\stackrel{p}{\to}X/G$:
\begin{align*}
	\cdots\longrightarrow H^{i-1}(X/G) \stackrel{\cup}{\longrightarrow} H^i(X/G)\stackrel{p^*_i}{\longrightarrow} H^i(X)\stackrel{\rho_i }{\longrightarrow}H^i(X/G)\stackrel{\cup}{\longrightarrow}H^{i+1}(X/G)\longrightarrow\cdots
\end{align*}
which begins with
\begin{align*}
	0\longrightarrow H^0(X/G)\stackrel{p^*_0}{\longrightarrow}H^0(X) \stackrel{\rho_0}{\longrightarrow}H^0(X/G)\stackrel{\cup}{\longrightarrow}H^1(X/G)
	\stackrel{p^*_1}{\longrightarrow} H^1(X)\longrightarrow\cdots.
\end{align*}
  By the exactness of  Gysin sequence,  the Euler class of the associated principal $G$-bundle $G\hookrightarrow X\to X/G$ must be  $c$.

For $0\leq i< \frac{n}{2}$, we have $H^{4i}(X/G) \cong \Z_2\oplus\cdots\oplus \Z_2$ ($2i+1$-copies), $ H^{4i+1}(X/G)\cong \Z_2\oplus\cdots \oplus\Z_2$ ($i+1$-copies), $ H^{4i+2}(X/G)\cong \Z_2\oplus\cdots \oplus\Z_2$ ($2i+2$-copies) and $ H^{4i+3}(X/G)\cong \Z_2\oplus\cdots \oplus\Z_2$ ($i+1$-copies) with bases
\begin{align*} &\{b^{i},ab^{i-1}c^2,a^2b^{i-1},a^3b^{i-2}c^2,\cdots,a^{2i-1}c^2,a^{2i}\},\\ 
&\{b^ic,a^2b^{i-1}c,\cdots, a^{2i-2}bc,a^{2i}c\},\\ &\{b^{i}c^2,ab^{i},a^2b^{i-1}c^2,a^3b^{i-1},\cdots,a^{2i}c^2,a^{2i+1}\}, \text{ and }\\ &\{ab^{i}c,a^3b^{i-1}c,\cdots,a^{2i-1}bc,a^{2i+1}c\},\end{align*}     respectively.
Again, by the  exactness of  Gysin sequence, we get  $\ker p_{4i}^*=\im \rho_{4i}
\cong\Z_2\oplus\cdots\oplus\Z_2$ ($i$-copies) with basis $\{a^{2j+1}b^{i-j-1}c^2| 0\leq j\leq i-1\}$, $\ker p_{4i+1}^*=
H^{4i+1}(X/G)$, $\ker p_{4i+2}^*=\im \rho_{4i+2}
\cong\Z_2\oplus\cdots\oplus\Z_2$ ($i+1$-copies) with basis $\{a^{2j}b^{i-j}c^2| 0\leq j\leq i\}$
 and $\ker p_{4i+3}^*=H^{4i+3}(X/G)$. This implies that $\im p ^*_{4i}\cong\im p ^*_{4i+2}\cong \Z_2\oplus\cdots\oplus\Z_2$ ($i+1$-copies) with bases $\{x^{2j}y^{i-j}|0\leq j\leq i\}$ and  $\{x^{2j+1}y^{i-j}|0\leq j\leq i\}$, respectively, where $p^*_2(a)=x,p_4^*(b)=y$. Clearly, $\im p^*_{4i+j}=\im\rho_{4i+j}=0,j=1,3$.
Thus,
  \begin{align*}
  	&H^{4i+1}(X)=H^{4i+3}(X)=0\\
	&H^{4i}(X)\cong  \Z_2\oplus\cdots\oplus\Z_2\, (2i+1\text{-copies}), \text{ and}\\
	&H^{4i+2}(X)\cong  \Z_2\oplus\cdots\oplus\Z_2\,(2i+2\text{-copies})
\end{align*} with  bases $\{x^{2j}y^{i-j},u_{4i}^{2l+1}|0\leq j\leq i, 0\leq l\leq i-1\}$ and  $\{x^{2j+1}y^{i-j},u_{4i+2}^{2j}|0\leq j\leq i\}$, respectively, where  $\rho_{4i}(u_{4i}^{2l+1})=a^{2l+1}b^{i-l-1}c^2,0\leq l\leq i-1 \text{ and }\rho_{4i+2}(u_{4i+2}^{2j})=a^{2j}b^{i-j}c^2,0\leq j\leq i$.

Let us assume that $n=2k$ for some $k>0$. In this case, for $2n\leq 4i+r\leq 4m,0\leq r\leq 3$, we have $H^{4i}(X/G) \cong \Z_2\oplus\cdots\oplus \Z_2$ ($2k+1$-copies), $ H^{4i+1}(X/G)\cong \Z_2\oplus\cdots \oplus\Z_2$ ($k+1$-copies), $ H^{4i+2}(X/G)\cong \Z_2\oplus\cdots \oplus\Z_2$ ($2k+1$-copies) and $ H^{4i+3}(X/G)\cong \Z_2\oplus\cdots \oplus\Z_2$ ($k$-copies) with bases
\begin{align*} &\{b^{i},ab^{i-1}c^2,a^2b^{i-1},a^3b^{i-2}c^2,\cdots, a^{2k-1}b^{i-k}c^2,a^{2k}b^{i-k}\},\\  
&\{b^ic,a^2b^{i-1}c,\cdots, a^{2k-2}b^{i-k+1}c,a^{2k}b^{i-k}c\},\\ &\{b^{i}c^2,ab^{i},a^2b^{i-1}c^2,a^3b^{i-1},\cdots,a^{2k-1}b^{i-k+1},a^{2k}b^{i-k}c^2\},\text{ and }\\ &\{ab^{i}c,a^3b^{i-1}c,\cdots,a^{2k-3}b^{i-k+2}c,a^{2k-1}b^{i-k+1}c\},\end{align*}     respectively.
Then by the  exactness of  Gysin sequence, we get  $\ker p_{4i}^*=\im \rho_{4i}
\cong\Z_2\oplus\cdots\oplus\Z_2$ ($k$-copies) with basis $\{a^{2j+1}b^{i-j-1}c^2| 0\leq j\leq k-1\}$, $\ker p_{4i+2}^*=\im \rho_{4i+2}
\cong\Z_2\oplus\cdots\oplus\Z_2$ ($k+1$-copies) with basis $\{a^{2j}b^{i-j}c^2| 0\leq j\leq k\}$ and  $\ker p_{4i+j}^*=H^{4i+j}(X/G),j=1,3$. This gives that $\im p ^*_{4i}\cong \Z_2\oplus\cdots\oplus\Z_2$ ($k+1$-copies) with basis $\{x^{2j}y^{i-j}|0\leq j\leq k\}$ and $\im p ^*_{4i+2}\cong \Z_2\oplus\cdots\oplus\Z_2$ ($k$-copies) with basis $\{x^{2j+1}y^{i-j}|0\leq j\leq k-1\}$. Also, $\im p ^*_{4i+j}=\im\rho_{4i+j}=0,j=1$ and $3$, Therefore,
\begin{align*}
	&H^{4i+1}(X)=H^{4i+3}(X)=0\\
	&H^{4i}(X)\cong  \Z_2\oplus\cdots\oplus\Z_2\, (n+1\text{-copies}), \text{ and}\\
	&H^{4i+2}(X)\cong  \Z_2\oplus\cdots\oplus\Z_2\,(n+1\text{-copies})
\end{align*} with  bases $\{x^{2j}y^{i-j},u_{4i}^{2l+1}|0\leq j\leq k, 0\leq l\leq k-1\}$ and  $\{x^{2j+1}y^{i-j},u_{4i+2}^{2l}|0\leq j\leq k-1,0\leq l\leq k\}$, respectively, where  $\rho_{4i}(u_{4i}^{2l+1})=a^{2l+1}b^{i-l-1}c^2,0\leq l\leq k-1\text{ and }\rho_{4i+2}(u_{4i+2}^{2l})=a^{2l}b^{i-l}c^2,0\leq l\leq k$.  
Now, for $4m< 4i+r\leq 4m+2n+2,0\leq r\leq 3$; we have $H^{4i}(X/G) \cong \Z_2\oplus\cdots\oplus \Z_2$ ($2m+n-2i+2$-copies), $ H^{4i+1}(X/G)\cong \Z_2\oplus\cdots \oplus\Z_2$ ($m+k-i+1$-copies), $ H^{4i+2}(X/G)\cong \Z_2\oplus\cdots \oplus\Z_2$ ($2m+n-2i+1$-copies) and $ H^{4i+3}(X/G)\cong \Z_2\oplus\cdots \oplus\Z_2$ ($m+k-i$-copies) with bases
\begin{align*}  &\{a^{2i-2m-1}b^{m}c^2,a^{2i-2m}b^{m},a^{2i-2m+1}b^{m-1}c^2,a^{2i-2m+2}b^{m-1},\cdots,a^{2k-1}b^{i-k}c^2,a^{2k}b^{i-k}\},\\  
&\{a^{2i-2m}b^mc,a^{2i-2m+2}b^{m-1}c,\cdots,a^{2k-2}b^{i-k+1}c,a^{2k}b^{i-k}c\},\\ &\{a^{2i-2m}b^{m}c^2,a^{2i-2m+1}b^{m},a^{2i-2m+2}b^{m-1}c^2,a^{2i-2m+3}b^{m-1},\cdots,a^{2k-1}b^{i-k+1},a^{2k}b^{i-k}c^2\},\text{ and }\\ &\{a^{2i-2m+1}b^{m}c,a^{2i-2m+3}b^{m-1}c,\cdots,a^{2k-3}b^{i-k+2}c,a^{2k-1}b^{i-k+1}c\},\end{align*}     respectively.
Then by the  exactness of  Gysin sequence, we get  $\ker p_{4i}^*=\im \rho_{4i}
\cong\Z_2\oplus\cdots\oplus\Z_2$ ($m+k-i+1$-copies) with basis $\{a^{2j+1}b^{i-j-1}c^2| i-m-1\leq j\leq k-1\}$, $\ker p_{4i+j}^*=H^{4i+j}(X/G),j=1,3$; and $\ker p_{4i+2}^*=\im \rho_{4i+2}
\cong\Z_2\oplus\cdots\oplus\Z_2$ ($m+k-i+1$-copies) with basis $\{a^{2j}b^{i-j}c^2| i-m\leq j\leq k\}$. This gives that $\im p ^*_{4i}\cong \Z_2\oplus\cdots\oplus\Z_2$ ($m+k-i+1$-copies) with basis $\{x^{2j}y^{i-j}|i-m\leq j\leq k\}$,  $\im p ^*_{4i+j}=\im\rho_{4i+j}=0,j=1,3$; and $\im p ^*_{4i+2}\cong \Z_2\oplus\cdots\oplus\Z_2$ ($m+k-i$-copies) with basis $\{x^{2j+1}y^{i-j}|i-m\leq j\leq k-1\}$. Therefore,
\begin{align*}
	&H^{4i+1}(X)=H^{4i+3}(X)=0\\
	&H^{4i}(X)\cong  \Z_2\oplus\cdots\oplus\Z_2\, (2m+n-2i+2\text{-copies}), \text{ and}\\
	&H^{4i+2}(X)\cong  \Z_2\oplus\cdots\oplus\Z_2\,(2m+n-2i+1\text{-copies})
\end{align*} with  bases $\{x^{2j}y^{i-j},u_{4i}^{2l+1}|i-m\leq j\leq k, i-m-1\leq l\leq k-1\}$ and  $\{x^{2j+1}y^{i-j},u_{4i+2}^{2l}|i-m\leq j\leq k-1,i-m\leq l\leq k\}$, respectively, where  $\rho_{4i}(u_{4i}^{2l+1})=a^{2l+1}b^{i-l-1}c^2,i-m-1\leq l\leq k-1 \text{ and }\rho_{4i+2}(u_{4i+2}^{2l})=a^{2l}b^{i-l}c^2,i-m\leq l\leq k$.

Next, we compute the cohomology algebra of $X$. 
Consider the Leray-Serre spectral sequence for the Borel fibration $X\stackrel{i}{\hookrightarrow}X_G\stackrel{\pi}{\rightarrow}B_G$. By the edge homomorphisms,  we get $d_r(1\otimes x)=0=d_r(1\otimes y)$ for all $r> 0$ and for all $i,j$;  $d_{r_{i,j}}(1\otimes u_j^i)\not=0$ for some $r_{i,j}>0$. Put $u_2^0=z$. So, $H^2(X)$ is generated by $\{x,z\}$. Clearly, $d_3(1\otimes z)=t^3\otimes 1$ and so $d_3(1\otimes xz)=t^3\otimes x\not=0$. This implies that
$xz=\alpha x^2+\beta y+ u_4^1,\alpha,\beta\in \Z_2$. So, 
$H^4(X)$  is generated by $\{x^2,y,xz\}$.
By the induction, as in Theorem \ref{theorem 4},  we get that for $1\leq 4i+r<2n,r=0,2$, $H^{4i}(X)$ and $H^{4i+2}(X)$ have bases $\{x^{2j}y^{i-j},x^{2l+1}y^{i-l-1}z|0\leq j\leq i,0\leq l\leq i-1\}$ and $\{x^{2j+1}y^{i-j},x^{2l}y^{i-l}z|0\leq j,l\leq i\}$, respectively. 
Similarly,  for $2n\leq 4i+r\leq 4m,r=0,2$, bases for $H^{4i}(X)$ and $H^{4i+2}(X)$ are  $\{x^{2j}y^{i-j},x^{2l+1}y^{i-l-1}z|0\leq j\leq k,0\leq l\leq k-1\}$ and $\{x^{2j+1}y^{i-j},x^{2l}y^{i-l}z|0\leq j\leq k-1,0\leq l\leq k\}$, respectively, and for  $4m< 4i+r\leq 4m+2n+2,r=0,2$, bases for $H^{4i}(X)$ and $H^{4i+2}(X)$ are $\{x^{2j}y^{i-j},x^{2l+1}y^{i-l-1}z|i-m\leq j\leq k,i-m-1\leq l\leq k-1\}$ and $\{x^{2j+1}y^{i-j},x^{2l}y^{i-l}z|i-m\leq j\leq k-1,i-m\leq l\leq k\}$, respectively.

Now, we take the case when $n=2k+1$ for some $k>0$. In this case, for $2n\leq 4i+r\leq 4m, 0\leq r\leq 3$, we have $H^{4i}(X/G) \cong \Z_2\oplus\cdots\oplus \Z_2$ ($2k+2$-copies), $ H^{4i+1}(X/G)\cong \Z_2\oplus\cdots \oplus\Z_2$ ($k+1$-copies), $ H^{4i+2}(X/G)\cong \Z_2\oplus\cdots \oplus\Z_2$ ($2k+2$-copies) and $ H^{4i+3}(X/G)\cong \Z_2\oplus\cdots \oplus\Z_2$ ($k+1$-copies) with bases
\begin{align*} &\{b^{i},ab^{i-1}c^2,a^2b^{i-1},a^3b^{i-2}c^2,\cdots,a^{2k}b^{i-k},a^{2k+1}b^{i-k-1}c^2\},\\  
&\{b^ic,a^2b^{i-1}c,\cdots,a^{2k-2}b^{i-k+1}c,a^{2k}b^{i-k}c\},\\ &\{b^{i}c^2,ab^{i},a^2b^{i-1}c^2,a^3b^{i-1},\cdots,a^{2k}b^{i-k}c^2,a^{2k+1}b^{i-k}\},\text{ and }\\ &\{ab^{i}c,a^3b^{i-1}c,\cdots,a^{2k-1}b^{i-k+1}c,a^{2k+1}b^{i-k}c\},\end{align*}     respectively.
Again, by the  exactness of  Gysin sequence, we get  $\ker p_{4i}^*=\im \rho_{4i}
\cong\Z_2\oplus\cdots\oplus\Z_2$ ($k+1$-copies) with basis $\{a^{2j+1}b^{i-j-1}c^2| 0\leq j\leq k\}$, $\ker p_{4i+2}^*=\im \rho_{4i+2}
\cong\Z_2\oplus\cdots\oplus\Z_2$ ($k+1$-copies) with basis $\{a^{2j}b^{i-j}c^2| 0\leq j\leq k\}$ and  $\ker p_{4i+j}^*=H^{4i+j}(X/G),j=1,3$. This gives that $\im p ^*_{4i}\cong \im p ^*_{4i+2}\cong \Z_2\oplus\cdots\oplus\Z_2$ ($k+1$-copies) with bases $\{x^{2j}y^{i-j}|0\leq j\leq k\}$ and  $\{x^{2j+1}y^{i-j}|0\leq j\leq k\}$, respectively. Note that $\im p ^*_{4i+j}=\im\rho_{4i+j}=0,j=1$ and $3$. Thus,
\begin{align*}
	&H^{4i+1}(X)=H^{4i+3}(X)=0\\
	&H^{4i}(X)\cong  \Z_2\oplus\cdots\oplus\Z_2\, (n+1\text{-copies}), \text{ and}\\
	&H^{4i+2}(X)\cong  \Z_2\oplus\cdots\oplus\Z_2\,(n+1\text{-copies})
\end{align*} with  bases $\{x^{2j}y^{i-j},u_{4i}^{2l+1}|0\leq j,l\leq k\}$ and  $\{x^{2j+1}y^{i-j},u_{4i+2}^{2l}|0\leq j,l\leq k\}$, respectively, where  $\rho_{4i}(u_{4i}^{2l+1})=a^{2l+1}b^{i-l-1}c^2,0\leq l\leq k\text{ and }\rho_{4i+2}(u_{4i+2}^{2l})=a^{2l}b^{i-l}c^2,0\leq l\leq k$.  
Now, for $4m< 4i+r\leq 4m+2n+2,0\leq r\leq 3$; we have $H^{4i}(X/G) \cong \Z_2\oplus\cdots\oplus \Z_2$ ($2m+n-2i+2$-copies), $ H^{4i+1}(X/G)\cong \Z_2\oplus\cdots \oplus\Z_2$ ($m+k-i+1$-copies), $ H^{4i+2}(X/G)\cong \Z_2\oplus\cdots \oplus\Z_2$ ($2m+n-2i+1$-copies) and $ H^{4i+3}(X/G)\cong \Z_2\oplus\cdots \oplus\Z_2$ ($m+k-i+1$-copies) with bases
\begin{align*}  &\{a^{2i-2m-1}b^{m}c^2,a^{2i-2m}b^{m},a^{2i-2m+1}b^{m-1}c^2,a^{2i-2m+2}b^{m-1},\cdots,a^{2k}b^{i-k},a^{2k+1}b^{i-k-1}c^2\},\\  
&\{a^{2i-2m}b^mc,a^{2i-2m+2}b^{m-1}c,\cdots, a^{2k-2}b^{i-k+1}c,a^{2k}b^{i-k}c\},\\ &\{a^{2i-2m}b^{m}c^2,a^{2i-2m+1}b^{m},a^{2i-2m+2}b^{m-1}c^2,a^{2i-2m+3}b^{m-1},\cdots,a^{2k}b^{i-k}c^2,a^{2k+1}b^{i-k}\},\text{ and }\\ &\{a^{2i-2m+1}b^{m}c,a^{2i-2m+3}b^{m-1}c,\cdots,a^{2k-1}b^{i-k+1}c,a^{2k+1}b^{i-k}c\},\end{align*}     respectively.
Now, $\ker p_{4i}^*=\im \rho_{4i}
\cong\Z_2\oplus\cdots\oplus\Z_2$ ($m+k-i+2$-copies) with basis $\{a^{2j+1}b^{i-j-1}c^2| i-m-1\leq j\leq k\}$, $\ker p_{4i+j}^*=H^{4i+j}(X/G),j=1,3$; and $\ker p_{4i+2}^*=\im \rho_{4i+2}
\cong\Z_2\oplus\cdots\oplus\Z_2$ ($m+k-i+1$-copies) with basis $\{a^{2j}b^{i-j}c^2| i-m\leq j\leq k\}$. Consequently, $\im p ^*_{4i}\cong\im p ^*_{4i+2}\cong \Z_2\oplus\cdots\oplus\Z_2$ ($m+k-i+1$-copies) with bases $\{x^{2j}y^{i-j}|i-m\leq j\leq k\}$  and $\{x^{2j+1}y^{i-j}|i-m\leq j\leq k\}$, respectively. Also,   $\im p ^*_{4i+j}=\im\rho_{4i+j}=0,j=1,3$. Therefore,
\begin{align*}
	&H^{4i+1}(X)=H^{4i+3}(X)=0\\
	&H^{4i}(X)\cong  \Z_2\oplus\cdots\oplus\Z_2\, (2m+n-2i+2\text{-copies}), \text{ and}\\
	&H^{4i+2}(X)\cong  \Z_2\oplus\cdots\oplus\Z_2\,(2m+n-2i+1\text{-copies})
\end{align*} with  bases $\{x^{2j}y^{i-j},u_{4i}^{2l+1}|i-m\leq j\leq k, i-m-1\leq l\leq k\}$ and  $\{x^{2j+1}y^{i-j},u_{4i+2}^{2l}|i-m\leq j,l\leq k\}$, respectively, where  $\rho_{4i}(u_{4i}^{2l+1})=a^{2l+1}b^{i-l-1}c^2,i-m-1\leq l\leq k \text{ and }\rho_{4i+2}(u_{4i+2}^{2l})=a^{2l}b^{i-l}c^2,i-m\leq l\leq k$.  Now, we discuss the  cohomology algebra of $X$. By the repeated use of induction,   
 for $\frac{n}{2}\leq i\leq m$, bases for $H^{4i}(X)$ and $H^{4i+2}(X)$ are  $\{x^{2j}y^{i-j},x^{2l+1}y^{i-l-1}z|0\leq j\leq k,0\leq l\leq k\}$ and $\{x^{2j+1}y^{i-j},x^{2l}y^{i-l}z|0\leq j\leq k,0\leq l\leq k\}$, respectively, and for  $m< i\leq m+\frac{n}{2}$, bases for $H^{4i}(X)$ and $H^{4i+2}(X)$ are
 $\{x^{2j}y^{i-j},x^{2l+1}y^{i-l-1}z|i-m\leq j\leq k,i-m-1\leq l\leq k\}$ and $\{x^{2j+1}y^{i-j},x^{2l}y^{i-l}z|i-m\leq j\leq k,i-m\leq l\leq k\}$, respectively.

Clearly, $x^{n+1}=y^{m+1}=0$. As $d_2(1\otimes z^2)=0$, $y^2=\alpha x^2+\beta y$ for some $\alpha,\beta\in \Z_2$. Thus, the cohomology algebra of $X$ is 
$$ \Z_2[x,y,z]/\langle x^{n+1},y^{m+1},z^2+\alpha x^2+\beta y\rangle,$$ where $\deg x=2,\deg y=4,\deg z=2$ and $\alpha,\beta \in \Z_2$.

Similarly, we get the same cohomology algebra for the case when $n\geq 2m.$
\end{proof}

Now,  we realize the possibilities of our main Theorems.
\begin{example}\label{Example 1}
Consider, the antipodal action of  $G=\Z_2$ on $\mathbb{S}^1$ and trivial action of $G$ on $\mathbb{CP}^n$. Then  the diagonal action of $G$ on $X=\mathbb{S}^1\times \mathbb{CP}^n$ is a free action with the orbit space $X/G=\mathbb{RP}^1\times\mathbb{CP}^n$. Also, the map $[x_0,x_1,\cdots, x_{2n},x_{2n+1}]\to [-x_1,x_0\cdots, -x_{2n+1},x_{2n}]$ gives a free action of $G$  on  $\mathbb{RP}^{2n+1}$ with the orbit space $X/G\sim_{2}\mathbb{RP}^1\times\mathbb{CP}^n $ \cite{Hemant2008}.  These examples realizes both possibilities of Theorem \ref{theorem 0}.
\end{example}

\begin{example}\label{Example 2}
	Consider, the antipodal action of  $G=\Z_2$ on $\mathbb{S}^2$ and  trivial action of $G$ on $\mathbb{HP}^n$. Then  the diagonal action of $G$ on $X=\mathbb{S}^2\times \mathbb{HP}^n$ is a free action with the orbit space $X/G=\mathbb{RP}^2\times\mathbb{HP}^n$. Also, the map $[z_0,z_1,\cdots, z_{n-1},z_n]\to [-\overline{z_1},\overline{z_0}\cdots, -\overline{z_n},\overline{z_{n-1}}]$ gives a free action of $G$  on  $\mathbb{CP}^{2n+1}$ with the orbit space $X/G\sim_{2}\mathbb{RP}^2\times\mathbb{HP}^n $ \cite{Hemant2008}.  These examples realizes both possibilities of Theorem \ref{theorem 2}.
\end{example}

 \begin{example}
Consider, the antipodal action of  $G=\Z_2$ on $\mathbb{S}^1$ and trivial action on $\mathbb{RP}^n\times \mathbb{CP}^m$. Then the diagonal action of $G$  on $\mathbb{S}^1\times \mathbb{RP}^n\times \mathbb{CP}^m$  is a free action with the orbit space  $\mathbb{RP}^1\times \mathbb{RP}^n\times \mathbb{CP}^m$.  This realizes possibility (i) of Theorem \ref{theorem 4} when $\alpha=\beta=0$.  	
 Now, we consider trival action of $G$ on $\mathbb{RP}^{n}$ and  free action of $G$ on $\mathbb{RP}^{2m+1}$ as defined in Example \ref{Example 1}. Then $G$ acts freely on  $X=\mathbb{RP}^n\times\mathbb{RP}^{2m+1}$  with the orbit space $X/G\sim_{2}\mathbb{RP}^n\times\mathbb{RP}^1\times  \mathbb{CP}^m$. This realizes possibility (i) of Theroem \ref{theorem 4} when $\alpha=0,\beta=1$.
 Also, if we take antipodal action of $G$ on $\mathbb{S}^n$ and trivial action on $\mathbb{RP}^1\times  \mathbb{CP}^m$ then $G$ acts freely on $\mathbb{S}^n\times\mathbb{RP}^1\times  \mathbb{CP}^m$  with the orbit space $\mathbb{RP}^n\times\mathbb{RP}^1\times  \mathbb{CP}^m$. This realizes possibilty (ii) of Theorem \ref{theorem 4} when $\alpha=0$.
 \end{example}
\begin{example}
	Consider, the antipodal action of  $G=\Z_2$ on $\mathbb{S}^2$ and  trivial action on $  \mathbb{CP}^n\times\mathbb{HP}^m$. Then the diagonal action of $G$  on $\mathbb{S}^2\times  \mathbb{CP}^n\times\mathbb{HP}^m$ is a free action with the orbit space  $\mathbb{RP}^2\times  \mathbb{CP}^n\times\mathbb{HP}^m$.  This realizes Theorem \ref{theorem 5} when $\alpha=\beta=0$.  	
Take trival action of $G$ on $\mathbb{CP}^{n}$ and free action of $G$ on $\mathbb{CP}^{2m+1}$ as defined in Example \ref{Example 2}. Then $G$ acts freely on  $X=\mathbb{CP}^n\times\mathbb{CP}^{2m+1}$  with the orbit space $X/G\sim_{2}\mathbb{CP}^n\times\mathbb{RP}^2\times  \mathbb{HP}^m$. This realizes  Theorem \ref{theorem 5} when $\alpha=0,\beta=1$.	
\end{example}

\section*{Acknowledgement}
We are thankful to the referee for their suggestions to improve the paper, correcting typos  and grammatical  errors.

   \bibliography{sn-bibliography}

\end{document}